\documentclass[12pt]{amsart}

\usepackage{amsmath,amssymb,amsbsy,amsfonts,amsthm,latexsym,
                     amsopn,amstext,amsxtra,euscript,amscd,mathrsfs}
\usepackage[colorlinks,linkcolor=blue,anchorcolor=blue,citecolor=blue]{hyperref}

\newtheorem{prop}{Proposition}
\newtheorem{lemma}[prop]{Lemma}

\newtheorem{theorem}[prop]{Theorem}

\def\cF{{\mathcal F}}

\def\cI{{\mathcal I}}

\def\cU{{\mathcal U}}
\def\cV{{\mathcal V}}

\def\F{{\mathbb F}}

\def\K{{\mathbb K}}

\def\Q{{\mathbb Q}}

\def\Z{{\mathbb Z}}

\def\fA{{\mathfrak A}}

\def\sB{{\mathscr B}}
\def\sC{{\mathscr C}}
\def\sD{{\mathscr D}}

\def\ep{{\mathbf{e}}_p}

\def\ssum{\mathop{\sum\, \sum}}

\def\\{\cr}
\def\({\left(}
\def\){\right)}
\def\[{\left[}
\def\]{\right]}
\def\<{\langle}
\def\>{\rangle}
\def\fl#1{\left\lfloor#1\right\rfloor}
\def\rf#1{\left\lceil#1\right\rceil}
\def\eps{\varepsilon}
\def\mand{\qquad\mbox{and}\qquad}

\def\deg{\operatorname{deg}}

\def\End{\operatorname{End}}

\def\exp{\operatorname{exp}}

\def\gcd{\operatorname{gcd}}

\def\GL{\operatorname{GL}}

\def\log{\operatorname{log}}

\begin{document}

\title[Lang--Trotter and Sato--Tate Conjectures on Average]{Lang--Trotter
and Sato--Tate Distributions in Single and Double Parametric Families of Elliptic Curves}

\author{Min Sha}
\address{School of Mathematics and Statistics, University of New South Wales,
 Sydney, NSW 2052, Australia}
\email{shamin2010@gmail.com}

\author{Igor E. Shparlinski}
\address{School of Mathematics and Statistics, University of New South Wales,
 Sydney, NSW 2052, Australia}
\email{igor.shparlinski@unsw.edu.au}

\date{}

\begin{abstract} We obtain new results concerning the Lang--Trotter conjectures on Frobenius traces and
Frobenius fields  over single and double parametric families of elliptic curves.
We also obtain similar results with respect to the Sato--Tate conjecture.
In particular, we improve a result of A.~C.~Cojocaru and the second
author (2008) towards the Lang--Trotter conjecture on average
for polynomially parameterized families of elliptic curves when the parameter
runs through a set of rational numbers of bounded height.
Some of the families we consider are much thinner than the ones previously
studied.
\end{abstract}

\maketitle

\section{Introduction}
\label{sec:intro}

\subsection{Background and motivation}

For polynomials $f(Z), g(Z) \in \Z[Z]$ satisfying
\begin{equation}
\label{eq:Nondeg}
\Delta(Z) \neq 0 \mand j(Z)  \not\in \Q,
\end{equation}
where
$$
\Delta(Z)   = -16 (4f(Z)^3 + 27g(Z)^2)\quad\text{and}\quad
j(Z)    = \frac{-1728(4f(Z))^{3}}{\Delta(Z)}
$$
are the {\it discriminant\/} and {\it $j$-invariant\/} respectively,
we consider
the   elliptic curve
\begin{equation}
\label{eq:Family AB}
E(Z) : \quad Y^2 = X^3 + f(Z)X + g(Z)
\end{equation}
over the function field $\Q(Z)$.
For a general background on elliptic curves we refer
to~\cite{Silv}.

Here we are interested in studying the specialisations
$E(t)$ of these curves on average over the parameter $t$
running through some interesting sets of integers or rational
numbers.
More precisely, motivated by the Lang--Trotter and Sato--Tate conjectures
we study the distributions of Frobenius traces, Frobenius fields and Frobenius angles of the reductions of  $E(t)$ modulo consecutive primes $p\le x$ for a growing parameter $x$, respectively.

Let us first introduce some standard notation.

Given an elliptic curve $E$ over $\Q$ we   denote by $E_p$
the reduction of $E$ modulo $p$. In particular, we use
$E_p(\F_p)$ to denote the group of $\F_p$-rational
points on $E_p$, where $\F_p$ is the finite field of
$p$ elements.  We always assume that the elements of $\F_p$
are represented by the set $\{0, \ldots, p-1\}$ and thus
we switch freely between the equations in $\F_p$ and congruences
modulo $p$.

For $a \in \Z$, we use
$\pi_{E} (a;x)$ to denote the number of primes $p \le x$ which do not
divide the conductor $N_E$  of $E$ and such that
$$
a_p(E) = a,
$$
where
$$
a_p(E) = p + 1 -\# E_p(\F_p)
$$
is the so-called \textit{Frobenius trace} of $E_p$.
We also set $a_p(E) =  0$ for $p\mid N_E$.

For a fixed imaginary quadratic field $\K$, we denote by $\pi_{E} (\K;x)$ the number
of primes $p\le x$ with $p\nmid N_E$ and such that
$$
a_p(E)\ne 0 \mand \Q\(\sqrt{a_p(E)^2-4p}\)=\K,
$$
where $\Q(\sqrt{a_p(E)^2-4p})$ is the so-called \textit{Frobenius field} of $E$ with respect to $p$.

Two celebrated Lang--Trotter conjectures~\cite{Lang} assert that
if $E$ is without complex multiplication (CM),
then
$$
\pi_{E}(a; x)  \sim  c(E,a)  \frac{\sqrt{x}}{\log x}
$$
as $x\to \infty$, for some constant  $c(E,a) \ge 0$
depending only on $E$ and $a$; if $E$ is without complex multiplication, then
$$
\pi_{E}(\K; x)  \sim  C(E,\K)  \frac{\sqrt{x}}{\log x}
$$
as $x\to \infty$, for some constant  $C(E,\K) \ge 0$
depending only on $E$ and $\K$.

However, the situation is quite different when $E$ has complex multiplication. For example, Deuring~\cite{Deuring} has showed that if $E$ has complex multiplication, then
\begin{equation}
\label{eq:Deuring}
\pi_{E}(0; x)  \sim \frac{1}{2} \cdot \frac{x}{\log x}.
\end{equation}
Besides, it is well-known that if $E$ is with complex multiplication, for any prime $p\nmid N_E$, we have
$$
\Q\(\sqrt{a_p(E)^2-4p}\)\simeq \End_{\bar{\Q}}(E)\otimes_{\Z} \Q,
$$
where $\End_{\bar{\Q}}(E)$ stands for the endomorphism ring of $E$; but if $E$ is without complex multiplication, there are infinitely many distinct such Frobenius fields as prime $p\nmid N_E$ varies.

Despite a series of interesting (conditional and unconditional)
recent
achievements, see~\cite{Coj,CojDav,CojShp,DavSm,Shp1,Shp4}
for surveys and some recent results,
these conjectures are widely open.

In addition,  by Hasse's bound, see~\cite{Silv},
 we can define the \textit{Frobenius angle} $\psi_p(E) \in [0, \pi]$ via the
identity
\begin{equation}
\label{eq:ST angle}
\cos \psi_p(E) = \frac{a_p(E)}{2\sqrt{p}}.
\end{equation}
For real numbers $0 \le \alpha < \beta \le \pi$, we define
the \emph{Sato--Tate  density}
\begin{equation}
\label{eq:ST dens}
\mu_{\tt ST}(\alpha,\beta) = \frac{2}{\pi}\int_\alpha^\beta
\sin^2\vartheta\, d \vartheta = \frac{2}{\pi}\int_{\cos \beta}^{\cos \alpha}
(1-z^2)^{1/2}\, d z.
\end{equation}

We denote by $\pi_{E}(\alpha,\beta;x)$ the number of
primes $p \le x$ (with $p \nmid N_E$) for which
$\psi_p(E) \in[\alpha, \beta]$. The  \emph{Sato--Tate  conjecture},
that has recently been settled  in the series of works of Barnet-Lamb,  Geraghty,
Harris, and Taylor~\cite{B-LGHT},
Clozel,  Harris and Taylor~\cite{CHT},
Harris,  Shepherd-Barron and  Taylor~\cite{HS-BT}, and Taylor~\cite{Taylor2008},
asserts that if $E$ is not a CM curve, then
\begin{equation}
\label{eq:ST conj}
\pi_{E}(\alpha,\beta;x) \sim
\mu_{\tt ST}(\alpha,\beta) \cdot\frac{x}{\log x}
\end{equation}
as $x \to \infty$. However, if $E$ is a CM curve, Deuring's result~\eqref{eq:Deuring} says that for half of primes $p$, the Frobenius angle  $\psi_p(E) = \pi$.

So, due to the lack of conclusive results towards the Lang--Trotter
conjectures, and also the lack of an explicit error term in the asymptotic
formula~\eqref{eq:ST conj}, it makes sense
to study $\pi_{E}(a; x)$, $\pi_{E}(\K; x)$
and $\pi_{E}(\alpha,\beta;x)$ on average over
some natural families of elliptic curves.

Here we continue  this line of research and in particular
introduce new natural families of curves, which are sometimes
much {\it thinner\/} than the ones previously studied in the literature.
We note that the  thinner the family the better the corresponding
result approximates the ultimate goal of obtaining
precise estimates for individual curves.

\subsection{Previously known results}

The idea of studying the properties of reduction $E_p$ for $p\le x$
on average  over a family of curves $E$ is due to Fouvry and  Murty~\cite{FoMu},
who have considered the average value of $\pi_{E}(0; x)$ and proved
 the Lang--Trotter conjecture on average
for the family of curves
\begin{equation}
\label{eq:Family uv}
E_{u,v}:\ Y^2 = X^3 + uX + v,
\end{equation}
where the integers $u$ and $v$ satisfy the inequalities
$|u| \le U$, $|v|\le V$. The results of~\cite{FoMu}
is nontrivial provided that
\begin{equation}
\label{eq:LT domain}
\min\{U,V\} > x^{1/2 + \eps} \mand UV > x^{3/2 + \eps}
\end{equation}
for some fixed  positive $\eps>0$, then,  on average, the Lang--Trotter
conjecture holds for such curves. Note that the case of  $\pi_{E}(0; x)$
corresponds to the distribution of so-called {\it supersingular primes\/}.
David and Pappalardi~\cite{DavPapp},  have extended
the result of~\cite{FoMu} to $\pi_{E}(a; x)$ with an arbitrary $a\in \Z$,
however under a more restrictive condition on $U$ and $V$
than that given by~\eqref{eq:LT domain}, namely  for
$ \min\{U,V\} > x^{1 + \eps}$. Finally,
Baier~\cite{Baier1} gives a full analogue
of the result of~\cite{FoMu} for any $a\in \Z$ and under
the same restriction~\eqref{eq:LT domain}; later Baier~\cite{Baier2}
also  replaces~\eqref{eq:LT domain} by the following condition
$$
 \min\{U,V\} > (\log x)^{60 + \eps} \quad \textrm{and} \quad x^{3/2}(\log x)^{10 + \eps}< UV < \exp(x^{1/8-\eps})
$$
when $a\ne 0$.
See also~\cite{BaJo} for a refined version of the Lang--Trotter
conjecture related to Frobenius traces with a uniform error term.

The Sato--Tate conjecture on average has also been
studied for the family~\eqref{eq:Family uv}, see~\cite{BaZha,BaSh}.
In particular, Banks and Shparlinski~\cite{BaSh} have shown that using
bounds of multiplicative character sums and the large sieve
inequality (instead of  the exponential
sum technique employed in~\cite{FoMu}), one can study the Sato--Tate conjecture
in a much wider range of $U$ and $V$ than that given
by~\eqref{eq:LT domain}. Namely, the results of~\cite{BaSh} are
nontrivial when
\begin{equation}
\label{eq:ST domain}
UV\ge x^{1 + \eps}\mand \min\{U,V\} \ge x^{ \eps}
\end{equation}
for some fixed  positive $\eps>0$, and the Sato--Tate conjecture is true on average for this family of elliptic curves.
The technique of~\cite{BaSh} has been used in several other
problems such as primality  or distribution of
values of $\# E_{u,v}(\F_p)$ in the domain, which is
similar to~\eqref{eq:ST domain}, see~\cite{CojDav,DavJ-U,Shp2}.

Results towards the Lang--Trotter and Sato--Tate conjectures
for more general families of the form $Y^2 = X^3 + f(u)X +g(v)$
with polynomials $f,g$ and integers $|u| \le U$, $|v|\le V$, are given in~\cite{Shp3}. Particularly, the conjectures are valid on average for these polynomial families of elliptic curves with restrictions on $U$ and $V$.

Furthermore, Cojocaru and Hall~\cite{CojHal} have considered the family
of curves~\eqref{eq:Family AB} and obtained an upper bound on the
average value of $\pi_{E(t)}(a; x)$
for the parameter $t$ that runs through the set of rational numbers
$$
\cF(T) = \{u/v \in \Q\ : \ \gcd(u, v) = 1, \ 1 \le
u,v \le T\},
$$
of height at most $T$.  For the size of $\cF(T)$, it is well known that
\begin{equation}\label{eq: Farey}
\# \cF(T)  \sim \frac{6}{\pi^2}T^2.
\end{equation}
as $T\to \infty$, see~\cite[Theorem~331]{HW}.
We recall that the set $\cF(T) \cap[0,1]$ is the well-known
set of {\it Farey fractions\/}.

Cojocaru and Shparlinski~\cite{CojShp}
have improved~\cite[Theorem~1.4]{CojHal} and obtained
a similar bound for the average value of $\pi_{E(t)}(a; x)$. Namely,
by~\cite[Theorem~2]{CojShp}, if the polynomials $f(Z), g(Z) \in \Z[Z]$
satisfy~\eqref{eq:Nondeg},
then, for any integer $a$, we have
\begin{equation}
\label{eq:pi a}
\sum_{\substack{t \in \cF(T)\\ \Delta(t) \ne 0}}
\pi_{E(t)}(a; x) \ll  T  x^{3/2 + o(1)} +
\left\{\begin{array}{ll}
 T^2 x^{3/4} &\quad\text{if}\  a \neq 0,\\
 T^2 x^{2/3} &\quad\text{if}\  a = 0;\\
\end{array}\right.
\end{equation}
and moreover for any imaginary quadratic field $\K$,
\begin{equation}
\label{eq:pi K}
\sum_{\substack{t \in \cF(T)\\ \Delta(t) \ne 0}}
\pi_{E(t)}(\K; x) \ll  T  x^{3/2 + o(1)} + T^2 x^{2/3}.
\end{equation}

Here we use the Landau symbols $O$ and $o$ and the Vinogradov symbol $\ll$. We recall that the assertions $A=O(B)$ and $A\ll B$ are both equivalent to the inequality $|A|\le cB$ with some absolute constant $c$, while $A=o(B)$ means that $A/B\to 0$. We also use the asymptotic notation $\sim$.
Throughout the paper the implied constants may depend on the polynomials $f(Z)$ and $g(Z)$ in~\eqref{eq:Family AB}.

\subsection{General outline of our results}
\label{sec:outline}

In this paper, we consider the Lang--Trotter and Sato--Tate conjectures on average for the polynomial family~\eqref{eq:Family AB} of elliptic curves when the variable $Z$ runs through sets of several different types.
More precisely, given a large positive parameter $T$, we consider
the case when $Z$ runs through $\cF(T)$ or a
 much ``thinner''  set  of $T$ consecutive
 integers, that is,
$$
 \cI(T) = \{1, \ldots, T\}.
$$
We believe that these are the first known  results that involve
one parametric family of curves, precisely with a parameter running through an interval of consecutive integers (note that $\cF(T)$ has the
structure and properties  of a two parametric set).

Furthermore, we also consider the case when $Z$ runs through the sums $ u +v$ (taken with multiplicities)
over all pairs $(u,v) \in \cU \times \cV$ for two subsets
$\cU, \cV \in \cI(T)$;
to the best of our knowledge,
results  in these settings, with arbitrary non-empty sets
$\cU$ and $\cV$, are completely new as well.

To derive our results we introduce several new ideas, such as using
a result of Michel~\cite[Proposition~1.1]{Mich} in a combination with a technique
of Niederreiter~\cite[Lemma~3]{Nied}. We also obtain several other
results of independent interest such as estimates of Section~\ref{sec:cong FT} for the number of solutions
of some congruences and equations with elements of $\cF(T)$.

We start with an improvement and generalisation of the bound~\eqref{eq:pi a},
and in fact give a  proof that is simpler than that of~\eqref{eq:pi a}.
More precisely, for an elliptic curve $E$ over $\Q$
and a sequence of integers $\fA = \{a_p\}$, supported on primes $p$, we
define $\pi_{E}(\fA; x)$ as the number of primes $p \le x$ which do not
divide the conductor $N_E$  of $E$ and such that
$$
a_p(E) = a_p.
$$

We say that $\fA$ is the {\it zero sequence\/} if $a_p=0$ for
every $p$, and $\fA$ is a \emph{constant sequence} if all $a_p$ equal to the same integer.
Note that if $a_p=a$ for all $p$, that is, $\fA$ is a constant sequence,
then $\pi_{E}(\fA; x)=\pi_{E}(a; x)$.
Here, one of the interesting choices of the sequence $\fA$ is with
$$
a_p = -\fl{2p^{1/2}},
$$
corresponding to curves with the
largest possible number of $\F_p$-rational points.

\subsection{Formulations of our results}
\label{sec:main}

We are now able to give exact formulations
of our results.

\begin{theorem}
\label{thm:L-T AB}
If the polynomials $f(Z), g(Z) \in \Z[Z]$
satisfy~\eqref{eq:Nondeg},
then for any sequence of integers $\fA = \{a_p\}$, we have
$$
\sum_{\substack{t \in \cF(T)\\ \Delta(t) \ne 0}}
\pi_{E(t)}(\fA; x) \ll
\left\{\begin{array}{ll}
 T  x^{11/8 + o(1)} + T^2 x^{7/8}  & \text{  for any $\fA$},\\
 T  x^{4/3 + o(1)} +  T^2 x^{5/6}  &\text{if $\fA$ is the zero sequence}.
\end{array}\right.
$$
\end{theorem}

The proof of Theorem~\ref{thm:L-T AB} is based on a
simple idea using the Cauchy inequality and then estimating
the second moment of the quantity of $R_{T,p}(w)$ (see Section~\ref{sec:not})
via a result of Ayyad, Cochrane and Zheng~\cite[Theorem~1]{ACZ}.
This gives a stronger result than the approach of~\cite{CojShp}
which is based on deriving an asymptotic formula for the average
deviation of $R_{T,p}(w)$ from its expected value (which also requires
to use the inclusion-exclusion formula).

Comparing this with~\eqref{eq:pi a}, we can see that if $a\ne0$
Theorem~\ref{thm:L-T AB} improves~\eqref{eq:pi a} and
remains nontrivial when $ x^{3/8+\varepsilon} \le T \le x^{5/8-\varepsilon}$ for small
$\varepsilon>0$.
If $a=0$ the same holds for $x^{1/3+\varepsilon}\le T \le x^{2/3-\varepsilon}$.
Furthermore, we note that~\eqref{eq:pi a} is nontrivial only when
$T\ge x^{1/2+\varepsilon}$, because the trivial upper bound is $O(T^2x)$.

We then consider the  very interesting and natural special case of
polynomials
\begin{equation}
\label{eq:j Poly}
f(Z) = 3Z(1728- Z) \mand g(Z) = 2Z(1728 - Z)^2
\end{equation}
for which one can verify that $j(Z) = Z$.
Thus for each specialisation $t\ne 0, 1728$, the
$j$-invariant of the curve  $E(t)$ equals $t$.
For this special case, we obtain a better bound than
that of Theorem~\ref{thm:L-T AB}.

\begin{theorem}
\label{thm:L-T J}
If the polynomials $f(Z), g(Z) \in \Z[Z]$
are given by~\eqref{eq:j Poly},
then for any sequence of integers $\fA = \{a_p\}$, we have
$$
\sum_{\substack{t \in \cF(T)\\ \Delta(t) \ne 0}}
\pi_{E(t)}(\fA; x) \ll  T  x^{5/4 + o(1)} +T^2 x^{3/4 + o(1)}  .
$$
\end{theorem}

We also get a non-trivial upper bound for the sum of $\pi_{E(r+s)}(\fA; x)$, where $r$ and $s$ run over $\cF(T)$ and $x^{1/4+\eps}\le T \le x^{1-\eps}$ for small $\eps>0$.

\begin{theorem}
\label{thm:L-T AB+}
Suppose that the polynomials $f(Z), g(Z) \in \Z[Z]$
satisfy~\eqref{eq:Nondeg}.
Then for any sequence of integers $\fA = \{a_p\}$, we have
$$
\sum_{\substack{r , s\in \cF(T) \\ \Delta(r+s) \ne 0}}
\pi_{E(r+s)}(\fA; x)
\ll T^5 + T^3x^{5/4+o(1)} + T^4x^{3/4+o(1)}.
$$
\end{theorem}

Now, we state a new result concerning the Lang-Trotter conjecture involving Frobenius fields.
\begin{theorem}
\label{thm:L-T K}
If the polynomials $f(Z), g(Z) \in \Z[Z]$
satisfy~\eqref{eq:Nondeg},
then for any imaginary quadratic field $\K$, we have
$$
\sum_{\substack{t \in \cF(T)\\ \Delta(t) \ne 0}}
\pi_{E(t)}(\K; x) \ll  T  x^{4/3 + o(1)} + T^2 x^{5/6} .
$$
\end{theorem}

Comparing this with~\eqref{eq:pi K}, we can see that
Theorem~\ref{thm:L-T K} improves~\eqref{eq:pi K} and remains
nontrivial when $x^{1/3+\varepsilon} \le T \le x^{2/3-\varepsilon}$ for  small $\varepsilon>0$.

The following result is the first study on the sum of $\pi_{E(r+s)}(\K; x)$ when $r$ and $s$ run over $\cF(T)$. Since the trivial bound is $T^4x$, this result is  nontrivial when $T\ge x^{1/6+\eps}$ for any $\eps>0$.

\begin{theorem}
\label{thm:L-T K+}
Suppose that the polynomials $f(Z), g(Z) \in \Z[Z]$
satisfy~\eqref{eq:Nondeg}.
Then for any imaginary quadratic field $\K$, we have
$$
\sum_{\substack{r , s\in \cF(T) \\ \Delta(r+s) \ne 0}}
\pi_{E(r+s)}(\K; x)
\ll  T^4x^{5/6}+ T^{2+o(1)} x^{4/3}.
$$
\end{theorem}

Unfortunately, currently there are no asymptotic formulas for the average value of
$\pi_{E(t)}(\alpha,\beta; x)$ (which is relevant to the Sato--Tate conjecture) when the parameter $t$
runs through $\cF(T)$. In particular the
arguments in the proof of Lemma~\ref{lem:Sin Farey} are not strong
enough for this.

Here, we consider this problem in another direction.
As usual, we use $\pi(x)$ to denote the number of primes $p\le x$.

\begin{theorem}
\label{thm:S-T Farey}
Suppose that the polynomials $f(Z), g(Z) \in \Z[Z]$
satisfy~\eqref{eq:Nondeg}, and for some $\eps>0$,
$$
x^{1/4+\eps} \le T \le x^{1-\eps}.
$$
Then for any real numbers $0\le \alpha<\beta \le \pi$, we have
$$
\frac{1}{(\#\cF(T))^2}\sum_{\substack{r , s\in \cF(T) \\ \Delta(r+s) \ne 0}}
\pi_{E(r+s)}(\alpha,\beta; x)=(\mu_{\tt ST}(\alpha,\beta)+O(x^{-\delta}))\pi(x),
$$

with  arbitrary real $\delta$ satisfying $0< \delta < \min \{\eps, 1/4\}$.
\end{theorem}

Note that in Theorem~\ref{thm:S-T Farey} it can be easy to drop the condition $T\le x^{1-\eps}$ and obtain a version of Theorem~\ref{thm:S-T Farey} under just one natural restriction $T\ge x^{1/4+\eps}$.
Since small values of $T$ are of our primal interest,
we have not attempted to do this.

We now recall that the common feature of the approaches of both~\cite{BaSh}
and~\cite{FoMu} is that they need two independently varying parameters
$u$ and $v$.
This has been a part of the motivation for
Cojocaru and Hall~\cite{CojHal} and Cojocaru and Shparlinski~\cite{CojShp}
to consider the family of curves~\eqref{eq:Family AB}. However,
even this family cannot be considered as a truly single
parametric family of curves, because the simple exclusion-inclusion
principle reduces a problem with the parameter $t \in \cF(T)$
to a series of problems with $t= u/v$, where $u$ and $v$
run independently through some intervals of consecutive integers.

To overcome this drawback, in~\cite{Shp3},
 the family of curves~\eqref{eq:Family AB}
has been studied for specialisations $t$ from the set
\begin{equation}
\label{eq:IT}
\cI(T) = \{1, \ldots, T\}
\end{equation}
of $T$ consecutive integers.
In particular, in~\cite[Theorem~15]{Shp3}, an asymptotic formula is
given for the average value of $\pi_{E(t)}(\alpha,\beta;x)$
over $t \in \cI(T)$,
provided that $T \ge x^{1/2+\varepsilon}$, thus providing
yet another form of the Sato--Tate conjecture on average.
This result is a first example of averaging over a
single parametric family of curves.
The proof of~\cite[Theorem~15]{Shp3}, amongst other things,
is based on a result of Michel~\cite{Mich}.
We note that unfortunately in~\cite[Lemma~9]{Shp3} a wrong reference is
given, a correct one is~\cite[Proposition~1.1]{Mich}.
Here we use a similar approach to estimate the
average value of $\pi_{E(t)}(\fA;x)$ over $t \in \cI(T)$,
that is, also for a single parametric family of curves,
which is related to the Lang--Trotter conjecture.

\begin{theorem}
\label{thm:L-T I}
If the polynomials $f(Z), g(Z) \in \Z[Z]$
satisfy~\eqref{eq:Nondeg},
then for any sequence of integers $\fA = \{a_p\}$, we have
$$
\sum_{\substack{t \in \cI(T)\\ \Delta(t) \ne 0}}
\pi_{E(t)}(\fA; x) \ll T^2 + T^{1/2}x^{5/4+o(1)}.
$$
\end{theorem}

In Theorem~\ref{thm:L-T I}, since the trivial upper bound is $Tx$, when $x^{1/2+\eps}< T <x^{1-\eps}$ for small $\eps>0$ the result is nontrivial.  We also have an analogue of Theorem~\ref{thm:L-T I}
over sum-sets.

\begin{theorem}
\label{thm:L-T Setsum}
If the polynomials $f(Z), g(Z) \in \Z[Z]$
satisfy~\eqref{eq:Nondeg},
then for any sequence of integers $\fA = \{a_p\}$ and
sets  of integer $\cU, \cV \subseteq \cI(T)$,
 we have
$$
 \sum_{\substack{u \in \cU, v\in \cV \\ \Delta(u+v) \ne 0}}
\pi_{E(u+v)}(\fA; x)   \ll T\#\cU \#\cV + (\#\cU \#\cV)^{3/4}x^{5/4}.
$$
\end{theorem}

As the above, in Theorem~\ref{thm:L-T Setsum} the trivial upper bound is $\#\cU \#\cV x$, so the result is nontrivial when $T < x^{1-\eps}$ and $\#\cU \#\cV > x^{1+\eps}$ for small $\eps>0$, which implies that $T> x^{(1+\eps)/2}$.

For the Lang--Trotter conjecture related to Frobenius fields, we get the following result when the parameter runs through $\cI(T)$. The result is nontrivial when $T>x^{2/3+\eps}$ for any $\eps>0$.

\begin{theorem}
\label{thm:L-T KI}
If the polynomials $f(Z), g(Z) \in \Z[Z]$
satisfy~\eqref{eq:Nondeg},
then for any imaginary quadratic field $\K$, we have
$$
\sum_{\substack{t \in \cI(T)\\ \Delta(t) \ne 0}}
\pi_{E(t)}(\K; x) \ll  T^{1/2} x^{4/3} + T x^{5/6} .
$$
\end{theorem}

We want to remark that since for any non-negative valued function $h(X)$, we have
$$
\sum_{r , s\in \cI(T)} h(r+s) \le T \sum_{t\in \cI(2T)} h(t),
$$
Theorem \ref{thm:L-T KI} implies the following upper bound
$$
\sum_{\substack{r , s\in \cI(T) \\ \Delta(r+s) \ne 0}}
\pi_{E(r+s)}(\K; x)\ll T^{3/2}x^{4/3} + T^2x^{5/6}.
$$

As mentioned before, in~\cite[Theorem~15]{Shp3}, an asymptotic formula is
given for the average value of $\pi_{E(t)}(\alpha,\beta;x)$
over $t \in \cI(T)$.
Here, we derive an analogue for the average value of $\pi_{E(u+v)}(\alpha,\beta; x)$, where $u,v$ run through two subsets $\cU, \cV$, respectively.

\begin{theorem}
\label{thm:S-T Setsum}
Suppose that the polynomials $f(Z), g(Z) \in \Z[Z]$
satisfy~\eqref{eq:Nondeg}, and
non-empty sets  of integer $\cU, \cV \subseteq \cI(T)$ are such that, for some $\eps>0$,
$$
\#\cU\#\cV\ge x^{1+\eps} \mand T\le x^{1-\eps}.
$$
Then for any real numbers $0\le \alpha<\beta \le \pi$, we have
$$
\frac{1}{\#\cU\#\cV}\sum_{\substack{u \in \cU, v\in \cV \\ \Delta(u+v) \ne 0}}
\pi_{E(u+v)}(\alpha,\beta; x)=\(\mu_{\tt ST}(\alpha,\beta)+O\(x^{-\eps/4}\)\)\pi(x).
$$
\end{theorem}

Note that in Theorem~\ref{thm:S-T Setsum}, since $T^2 \ge \#\cU\#\cV\ge x^{1+\eps}$,
we have $T\ge x^{(1+\eps)/2}$.

We remark that in this paper we often replace summation over primes by summation
over all integers. Thus some terms in the above bounds can be improved
by a small power of $\log x$.

\section{Preliminaries}
\label{Preliminary}

\subsection{Notation and general remarks}
\label{sec:not}

Throughout the paper, $p$ always denotes a prime number. For  $t \in \Q$, let $N(t)$ denote the conductor of the  specialisation of
$E(Z)$ at $Z = t$. We always consider rational numbers in the form of irreducible fractions.

Note that for $t\in \Q$, $\Delta(t)$ may be a rational number. However, we know that the elliptic curve $E(t)$ has good reduction at prime $p$ if and only if $p$ does not divide both the numerator and denominator of $\Delta(t)$; see~\cite[Chapter VII, Proposition 5.1 (a)]{Silv}.
So, we can say that for any prime $p$, $p \nmid N(t)$ (that is, $E(t)$ has good reduction at $p$) if and only if $\Delta(t)\not\equiv 0 \pmod p$ (certainly, it first requires that $p$ does not divide the denominator of $\Delta(t)$).

We define
\begin{equation}
\label{eq:PF}
P_{\cF}= \# \{ \textrm{$p$ prime}: \exists\, u/v\in \Q, p \mid v, p \nmid N(u/v) \}.
\end{equation}
Since $p \nmid N(u/v)$, we have $\Delta(u/v)\not\equiv 0 \pmod p$, which requires that $p$ does not divide the denominator of $\Delta(u/v)$. Noticing the form of $\Delta(u/v)$ and $p \mid v$, we can see that $P_{\cF}$ is upper bounded by a constant which only depends on the polynomials $f(Z), g(Z)$. For example, if $\deg f >\deg g$, then $P_{\cF}$ is not greater than the number of prime divisors of $2a_f$, where $a_f$ is the leading coefficient of $f(Z)$, because such a prime $p$ must divide $2a_f$.

For an integer $w$, we denote by $R_{T,p}(w)$ the number
of fractions $u/v \in \cF(T)$ with $\gcd(v,p)=1$
and $u/v \equiv w \pmod p$. In particular, we immediately derive
 the inequality
\begin{align}
\label{eq:pi R}
\sum_{\substack{t \in \cF(T)\\ \Delta(t) \ne 0}}
\pi_{E(t)}(\fA; x)
&\le P_{\cF}T^2 +  \sum_{\substack{t=u/v \in \cF(T)\\ \Delta(t) \ne 0}}
\sum_{\substack{p\le x \\ p\nmid v, p\nmid N(t) \\ a_{t,p}=a_p}} 1 \\
&\le P_{\cF}T^2 +
\sum_{p\le x}
 \sum_{\substack{0\leq w \leq p-1\\
\Delta(w)  \not \equiv 0 \pmod p  \\ a_{w,p}
= a_p }} R_{T,p}(w),\notag
\end{align}
where to simplify the notation we denote
\begin{equation}
\label{eq:def awp}
a_{w,p}=a_p(E(w)).
\end{equation}
We want to indicate that the treatment in~\eqref{eq:pi R} is an improvement of the
inequality used in~\cite[Section~3.2]{CojShp} (at the bottom of~\cite[Page~1982]{CojShp}), 
however this does not affect the final result of~\cite[Theorem~2]{CojShp}.

\subsection{Some congruences with traces}

The following estimate follows  immediately from \eqref{eq:pi R}.

\begin{lemma}
\label{lem:ap R} If the polynomials $f(Z), g(Z) \in \Z[Z]$
satisfy~\eqref{eq:Nondeg},
then for any sequence of integers $\fA = \{a_p\}$  and prime $\ell$, we have
$$
\sum_{\substack{t \in \cF(T)\\ \Delta(t) \ne 0}}
\pi_{E(t)}(\fA; x)
\le P_{\cF}T^2 + \sum_{p\le x}
 \sum_{\substack{0\leq w \leq p-1\\
\Delta(w)  \not \equiv 0 \pmod p  \\ a_{w,p}
\equiv a_p \pmod \ell}} R_{T,p}(w).
$$
\end{lemma}

Next we need the following two  bounds that have been  obtained in the
proof of~\cite[Theorem~2]{CojShp} (see the middle and the bottom of~\cite[Page~1983]{CojShp}, and~\cite[Equation~(8)]{CojShp} respectively) from an effective
version of the {\it Chebotarev theorem} given by Murty and
Scherk~\cite[Theorem~2]{MurSch}, see also~\cite[Theorem~1.2]{CojHal}.

\begin{lemma}
\label{lem:ap cong1} If the polynomials $f(Z), g(Z) \in \Z[Z]$
satisfy~\eqref{eq:Nondeg},
then for any integer $a$ and prime $\ell\ge 17$ and $\ell\ne p$, we have
$$
 \sum_{\substack{0\leq w \leq p-1\\
\Delta(w) \not \equiv 0 \pmod p  \\ a_{w,p}
\equiv  a \pmod \ell}} 1
=
\frac{p}{\ell} +
\left\{\begin{array}{ll}
O(\ell p^{1/2}) &\quad\text{if}\  a \neq 0,\\
 O(\ell^{1/2} p^{1/2}) &\quad\text{if}\  a = 0,
\end{array}\right.
$$
where the implied constants are independent of $a,p$ and $\ell$.
\end{lemma}

\begin{lemma}
\label{lem:ap cong2} If the polynomials $f(Z), g(Z) \in \Z[Z]$
satisfy~\eqref{eq:Nondeg},
then for any prime $\ell\ge 17$ and $\ell\ne p$, and any imaginary quadratic field $\K$, we have
$$
 \sum_{\substack{0\leq w \leq p-1\\
\Delta(w) \not \equiv 0 \pmod p  \\ a_{w,p}
\not\equiv  0 \pmod p \\ \Q(\sqrt{a_{w,p}^2-4p})=\K} } 1
=
\frac{p}{\ell} +
 O(\ell^{1/2} p^{1/2}),
$$
where  the implied constants are independent of $\K$, $p$ and $\ell$.
\end{lemma}

\subsection{Some congruences with elements of $\cF(T)$}
\label{sec:cong FT}

We first prove the following estimate on the average multiplicity
of values in the reduction of $\cF(T)$ modulo $p$, which is used several times later on.

\begin{lemma}
\label{lem:B_{T,p}}
For any prime $p$, define
\begin{equation*}
\begin{split}
Q_{T,p}=\# \{(u_1/v_1,u_2/v_2)\in \cF(T)\times\cF(T): &\gcd(v_1v_2,p)=1, \\
& u_1/v_1\equiv u_2/v_2 \pmod p\}.
\end{split}
\end{equation*}
Then, we have
$$
Q_{T,p} \ll T^4/p + T^2(\log p)^2 = T^4/p + T^2p^{o(1)},
$$
where the implied constant is independent of $p$ and $T$.
\end{lemma}

\begin{proof}
Dropping the condition
$$
\gcd(v_1v_2,p)=\gcd(u_1, v_1)=\gcd(u_2, v_2) = 1,
$$
we see that  $Q_{T,p}$ does not exceed
 the number of solutions to the congruence
$$
u_1v_2 \equiv u_2v_1 \pmod p,\qquad
1 \le u_1,u_2,v_1,v_2 \le T,
$$
which has been estimated as $O(T^4/p + T^2(\log p)^2)$
by Ayyad,  Cochrane and Zheng~\cite[Theorem~1]{ACZ} when $T<p$.
Obviously, by fixing three variables and varying the remaining variable, when $p\le T$ the number of such solutions is at most $2T^4/p$ .
So, we have
$$
Q_{T,p}\ll T^4/p + T^2(\log p)^2 =T^4/p + T^{2} p^{o(1)},
$$
where  the implied constant is independent of $p$ and $T$.
\end{proof}

We now need an additive analogue of Lemma~\ref{lem:B_{T,p}}. Namely, we need
an upper bound on  the number  $V_{T,p}$ of solutions to the congruence
\begin{equation}
\label{eq:add energy FTp}
\begin{split}
u_1/v_1 + u_2/v_2  &\equiv u_3/v_3 + u_4/v_4 \pmod p,\\
u_i/v_i \in \cF(T),\ i=1,&2,3,4, \qquad
\gcd(v_1v_2v_3v_4,p)=1.
\end{split}
\end{equation}

Trivially we have $V_{T,p} \ll T^8/p + T^7$. Using bounds of exponential
sums  with Farey fractions from~\cite{Shp0},  one can get an essentially
optimal bound.
We also denote $\ep(z) = \exp(2 \pi i z/p)$.

\begin{lemma}
\label{lem:ExpSum FN}
For any prime $p$, we have
$$
\max_{a \in \F_p^*} \left| \sum_{u/v \in \cF(T)} \ep(au/v)\right| \le T
(Tp)^{o(1)}.
$$
\end{lemma}

\begin{proof} The desired result looks similar to~\cite[Theorem~1]{Shp0}
taken with $m = p$. However in~\cite{Shp0} the set $\cF(T)$ is defined in a more traditional
way with the additional condition $u < v$ (that is, $\cF(T) \subseteq [0,1]$
in the definition of~\cite{Shp0}). So we give here a short proof which relies on
the bound in~\cite[Lemma~3]{Shp0}.
Namely, let $U, V\ge 1$ be arbitrary integers and let for each $v$ we are
given two integers $U_v > L_v $ with $0\le L_v < p$ and $U_v \le U$. Then by~\cite[Lemma~3]{Shp0}, taken with $m=p$,
we have
\begin{equation}
\label{eq:expsum Farey}
\max_{a \in \F_p^*} \left|
 \sum_{\substack{v=1\\ \gcd(v,p)=1}}^V \sum_{u=L_v+1}^{U_v} \ep(au/v)\right| \le (U+V)
(Vp)^{o(1)}.
\end{equation}
Now, for an integer $d \ge 1$ we use $\mu(d)$ to denote the M\"obius function.
We recall that $\mu(1) = 1$, $\mu(d) = 0$ if $d \ge 2$ is not square-free,  and $\mu(d) = (-1)^{\omega(d)}$  otherwise, where $\omega(d)$ is the number of prime
divisors of $d$.
Then by the inclusion-exclusion principle,
\begin{equation*}
\begin{split}
\sum_{u/v \in \cF(T)} \ep(au/v) & =
\sum_{d=1}^T \mu(d) \sum_{\substack{v=1\\\gcd(v,p)=1\\d \mid v} }^T
\sum_{\substack{u=1\\ d \mid u}}^T \ep\(a u/v\)\\
& =   \sum_{d=1}^T \mu(d) \sum_{\substack{v=1\\ \gcd(v,p)=1}}^{\fl{T/d}}
\sum_{u= 1}^{\fl{T/d}} \ep\(a u/v\).
\end{split}
\end{equation*}
Now, for each $d =1, \ldots, T$ we apply~\eqref{eq:expsum Farey}
to see that each inner sum is at most $Td^{-1} (Tp)^{o(1)}$.
The result now follows.
\end{proof}

We are now ready to estimate $V_{T,p}$.

\begin{lemma}
\label{lem:V_{T,p}}
For any prime $p$, we have
$$
 V_{T,p} = \frac{\(\#\cF(T)\)^4}{p} + O(T^4 (Tp)^{o(1)}).
$$
\end{lemma}

\begin{proof}
Using the orthogonality of the exponential function, we write
$$
 V_{T,p}
  = \ssum_{\substack{u_i/v_i \in \cF(T) \\ i=1,2,3,4}} \, \frac{1}{p}
\sum_{a=0}^{p-1} \ep\(a\(u_1/v_1 + u_2/v_2 - u_3/v_3 - u_4/v_4\)\).
$$
Changing the order of summation  and also noticing that
$|z|^2 = z \overline z$, we obtain
$$
 V_{T,p} =  \frac{1}{p} \sum_{a=0}^{p-1}  \left| \sum_{u/v \in \cF(T)} \ep(au/v)\right|^4.
$$
Now, the contribution from $a=0$ gives the main term $\(\#\cF(T)\)^4/p$,
while for other sums we apply Lemma~\ref{lem:ExpSum FN}, which
concludes the proof.
\end{proof}

\subsection{Preparations for distribution of angles}
\label{sec:preparation}

Now, we introduce a direct consequence of a result of Niederreiter~\cite[Lemma~3]{Nied},
which is one of our key tools.
For $m$ arbitrary elements $w_1,\ldots,w_m$ lying in the interval $[-1,1]$ (not necessarily distinct) and an arbitrary subinterval $J$
of $[-1,1]$, let $A(J;m)$ be the number of integers $i$, $1\le i \le m$, with $w_i \in J$. For any $-1\le a < b \le 1$, define the  function
$$
G(a,b)=\frac{2}{\pi} \int_{a}^{b} (1-z^2)^{1/2} \, d z.
$$
We also recall the Chebyshev polynomials $U_n$ of the second kind, on $[-1,1]$ they are defined by
$$
U_n(z) = \frac{\sin((n+1)\arccos z)}{(1-z^2)^{1/2}} \quad \textrm{for $z\in [-1,1]$},
$$
where $n$ is a nonnegative integer. In particular, for $\vartheta \in [0,\pi]$, we have
$$
U_n(\cos \vartheta) = \frac{\sin((n+1)\vartheta)}{\sin \vartheta}.
$$

\begin{lemma}
\label{lem:Nie}
For any integer $k\ge 1$, we have
$$
\max_{-1\le a < b \le 1} \left| A([a,b];m) - mG(a,b) \right|
\ll \frac{m}{k} + \sum_{n=1}^{k}\frac{1}{n} \left| \sum_{i=1}^{m} U_n(w_i) \right|.
$$
\end{lemma}
\begin{proof}
Note that for any $-1\le a < b \le 1$, we have
\begin{align*}
& A([a,b];m) - mG(a,b) \\
& \qquad = \( A([-1,b];m) - mG(-1,b) \) - \( A([-1,a);m) - mG(-1,a) \).
\end{align*}
For any odd positive integer $\kappa$, it follows directly from~\cite[Lemma~3]{Nied} that
\begin{align*}
& | A([a,b];m) - mG(a,b) | \\
& \qquad < \frac{16m}{0.362\cdot \pi \kappa +4} + \frac{2(4\kappa -3)}{0.362\cdot\pi \kappa +2\pi} \sum_{n=1}^{2\kappa-1}\frac{n+1}{n(n+2)} \left| \sum_{i=1}^{m} U_n(w_i) \right|.
\end{align*}
The desired result now follows by varying the value of $\kappa$ according to $k$.
Here, one ought to notice the symbol ``$\ll$'' we use in the result.
\end{proof}

\subsection{Distribution of angles over $\cF(T)$}
\label{sec:Farey}

We now consider the angles  $\psi_p(E(t))$
that are given by~\eqref{eq:ST angle}.

Michel~\cite[Proposition~1.1]{Mich} gives
the following bound on the weighed sums with
 the angles  $\psi_p(E(t))$  for single parametric polynomial families of
curves, where the sums is also twisted by additive characters.

 We recall the notation
$\ep(z) = \exp(2 \pi i z/p)$ from Section~\ref{sec:cong FT}.

\begin{lemma}
\label{lem:Mich bound} If the polynomials $f(Z), g(Z) \in \Z[Z]$
satisfy~\eqref{eq:Nondeg}, we have
$$
\sum_{\substack{w \in \F_p\\
\Delta(w) \not \equiv 0 \pmod p}}
\frac{\sin\((n+1)\psi_p(E(w))\)}{\sin\( \psi_p(E(w))\)} \ep\(mw\) \ll
np^{1/2},
$$
uniformly over  all
integers $m$ and $n\ge 1$.
\end{lemma}

The following result is a direct application of Lemma~\ref{lem:Mich bound}.

\begin{lemma}
\label{lem:Sin Farey}
If the polynomials $f(Z), g(Z) \in \Z[Z]$
satisfy~\eqref{eq:Nondeg},
then for any prime $p$,
 we have
$$
\sum_{\substack{r, s \in \cF(T)\\
\Delta(r+s) \not \equiv 0 \pmod p}}
\frac{\sin ((n+1)\psi_p(E(r+s)))}{\sin (\psi_p(E(r+s)))}  \ll
nT^2p^{1/2+o(1)}+nT^4p^{-1/2},
$$
uniformly over all integers $n\ge 1$.
\end{lemma}

\begin{proof} Using the orthogonality of the exponential function, we write 
\begin{equation*}
\begin{split}
&\sum_{\substack{r, s \in \cF(T)\\
\Delta(r+s) \not \equiv 0 \pmod p}}
\frac{\sin ((n+1)\psi_p(E(r+s)))}{\sin (\psi_p(E(r+s)))}\\
& \qquad \quad =\sum_{\substack{w \in \F_p\\
\Delta(w) \not \equiv 0 \pmod p}}
\frac{\sin\((n+1)\psi_p(E(w))\)}{\sin\( \psi_p(E(w))\)}\\
& \qquad \qquad \qquad
\sum_{\substack{u_1/v_1\in \cF(T), \,\gcd(v_1,p)=1 \\ u_2/v_2\in \cF(T), \,\gcd(v_2,p)=1}}
\frac{1}{p} \sum_{m=0}^{p-1}\ep(m(w - u_1/v_1-u_2/v_2)) \\
& \qquad \qquad \qquad + O(nT^3(T/p + 1)), 
\end{split}
\end{equation*}
where the last term comes from the exceptional case with $p \mid v_1v_2$. 
So changing the order of summation we obtain:
\begin{equation*}
\begin{split}
&\sum_{\substack{r, s \in \cF(T)\\
\Delta(r+s) \not \equiv 0 \pmod p}}
\frac{\sin ((n+1)\psi_p(E(r+s)))}{\sin (\psi_p(E(r+s)))}\\
& \qquad \quad= \frac{1}{p} \sum_{m=0}^{p-1}
  \sum_{\substack{w \in \F_p\\
\Delta(w) \not \equiv 0 \pmod p}}
\frac{\sin\((n+1)\psi_p(E(w))\)}{\sin\( \psi_p(E(w))\)}
\ep(mw) \\
& \qquad \qquad\qquad \qquad
\sum_{\substack{u_1/v_1\in \cF(T) \\ \gcd(v_1,p)=1}}  \ep(-mu_1/v_1)
\sum_{\substack{u_2/v_2\in \cF(T) \\ \gcd(v_2,p)=1}}  \ep(-mu_2/v_2) .
\end{split}
\end{equation*}
Using Lemma~\ref{lem:Mich bound}, we have
\begin{equation*}
\begin{split}
&\sum_{\substack{r, s \in \cF(T)\\
\Delta(r+s) \not \equiv 0 \pmod p}}
\frac{\sin ((n+1)\psi_p(E(r+s)))}{\sin (\psi_p(E(r+s)))}\\
&\qquad \ll n p^{-1/2} \sum_{m=0}^{p-1}
\left|\sum_{\substack{u_1/v_1\in \cF(T) \\ \gcd(v_1,p)=1}}  \ep(-mu_1/v_1)  \right|
\left|\sum_{\substack{u_2/v_2\in \cF(T) \\ \gcd(v_2,p)=1}}  \ep(-mu_2/v_2) \right|.
\end{split}
\end{equation*}
It now remains to apply the Cauchy inequality and note  that
\begin{align*}
& \sum_{m=0}^{p-1}
\left|\sum_{\substack{u/v\in \cF(T) \\ \gcd(v,p)=1}}  \ep(-mu/v)  \right|^2 \\
& \quad =\sum_{m=0}^{p-1}
\sum_{\substack{u_1/v_1\in \cF(T) \\ u_2/v_2\in \cF(T) \\ \gcd(v_1v_2,p)=1}}  \ep\(m(u_2/v_2-u_1/v_1)\)   \\
& \quad =
\sum_{\substack{u_1/v_1\in \cF(T) \\ u_2/v_2\in \cF(T) \\ \gcd(v_1v_2,p)=1}}  \sum_{m=0}^{p-1} \ep\(m(u_2/v_2-u_1/v_1)\)
\ll T^2p^{1+o(1)}+T^4,
\end{align*}
which follows from the orthogonality of the exponential function and Lemma~\ref{lem:B_{T,p}}.
\end{proof}

Now, we define $\sB_{f,g,p}(\cF(T);\alpha,\beta)$ as
the number of pairs $(r,s) \in \cF(T)\times \cF(T)$  with $\Delta(r+s) \not \equiv 0 \pmod p$
such that
$$
\alpha \le \psi_p(E(r+s))\le \beta.
$$

Now, combining Lemma~\ref{lem:Nie} with Lemma~\ref{lem:Sin Farey}  we derive the following result. Note that here we assume that the prime $p$ is greater than $T$. Since we prefer small values of $T$, this assumption is reasonable.

\begin{lemma}
\label{lem:ST Farey}
If the polynomials $f(Z), g(Z) \in \Z[Z]$
satisfy~\eqref{eq:Nondeg}, then for
any prime $p>T$, we have
\begin{align*}
 \max_{0 \le \alpha < \beta \le \pi} \left|\sB_{f,g,p}(\cF(T) ;\alpha,\beta) -
\mu_{\tt ST}(\alpha,\beta) (\# \cF(T))^2   \right| & \\
 \ll T^3p^{1/4+o(1)} &+T^4p^{-1/4+o(1)} .
\end{align*}
\end{lemma}

\begin{proof}
Obviously, since $p>T$, we have
\begin{equation*}
\#\{(r,s) \in \cF(T)\times \cF(T):~\Delta(r+s) \equiv 0 \pmod p\}
\ll T^3.
\end{equation*}

We now associate to each pair $(r,s) \in \sB_{f,g,p}(\cF(T);\alpha,\beta)$
a value $\cos \psi_p(E(r+s))$. This enables us to apply Lemma~\ref{lem:Nie}.
So, by Lemma~\ref{lem:Nie}, for any positive integer $k$, we have
\begin{equation*}
\begin{split}
\max_{0 \le \alpha < \beta \le \pi} & \left|\sB_{f,g,p}(\cF(T) ;\alpha,\beta) -
\mu_{\tt ST}(\alpha,\beta) (\# \cF(T))^2   \right|\\
& \ll T^3 + \frac{(\# \cF(T))^2}{k} \\
& \qquad + \sum_{n=1}^{k} \frac{1}{n}\left|\sum_{\substack{r, s \in \cF(T)\\
\Delta(r+s) \not \equiv 0 \pmod p}}
\frac{\sin\((n+1)\psi_p(E(r+s))\)}{\sin\( \psi_p(E(r+s))\)}  \right|.
\end{split}
\end{equation*}
Here, the reason why the term $T^3$  appears in the above inequality is that
the pairs $(r,s)$ satisfying $\Delta(r+s) \equiv 0 \pmod p$ are not counted in $\sB_{f,g,p}(\cF(T) ;\alpha,\beta)$.

Thus, by~\eqref{eq: Farey} and Lemma~\ref{lem:Sin Farey}, we get
\begin{equation}
\label{eq:Prelim B}
\begin{split}
\max_{0 \le \alpha < \beta \le \pi} &  \left|\sB_{f,g,p}(\cF(T) ;\alpha,\beta) -
\mu_{\tt ST}(\alpha,\beta) (\# \cF(T))^2   \right|\\
& \qquad  \ll T^3 + \frac{T^4}{k} +   k T^2p^{1/2+o(1)} +kT^4p^{-1/2}\\
& \qquad \ll    \frac{T^4}{k} +   k T^2p^{1/2+o(1)} +kT^4p^{-1/2}.
\end{split}
\end{equation}

Clearly, we can assume that $T \ge p^{1/4}$ as otherwise the result is
weaker than the trivial bound $O(T^4)$.

Now, for $p^{1/2} \ge T \ge p^{1/4}$   we take  $k =\rf{p^{-1/4}T}$ to balance the first two terms in~\eqref{eq:Prelim B} and derive
\begin{equation}
\label{eq:small T}
\begin{split}
 \max_{0 \le \alpha < \beta \le \pi}& \left|\sB_{f,g,p}(\cF(T) ;\alpha,\beta) -
\mu_{\tt ST}(\alpha,\beta) (\# \cF(T))^2   \right|  \\
& \qquad \qquad  \ll  T^3p^{1/4+o(1)}+T^5p^{-3/4} \le  T^3p^{1/4+o(1)},
\end{split}
\end{equation}

For $T\ge p^{1/2}$ we take  $k =\rf{p^{1/4}}$ to balance the
first  and the third terms in~\eqref{eq:Prelim B} and derive
\begin{equation}
\label{eq:large T}
\begin{split}
 \max_{0 \le \alpha < \beta \le \pi}& \left|\sB_{f,g,p}(\cF(T) ;\alpha,\beta) -
\mu_{\tt ST}(\alpha,\beta) (\# \cF(T))^2   \right|  \\
& \qquad \qquad  \ll  T^2p^{3/4+o(1)}+T^4p^{-1/4} \le T^4p^{-1/4+o(1)} .
\end{split}
\end{equation}

Finally, noticing that $T^4p^{-1/4} \le T^3p^{1/4}$ is equivalent to
$T \le p^{1/2}$, we  see that in both cases the bounds~\eqref{eq:small T}
and~\eqref{eq:large T} can be combined in one bound
\begin{equation*}
\begin{split}
 \max_{0 \le \alpha < \beta \le \pi} \left|\sB_{f,g,p}(\cF(T) ;\alpha,\beta) -
\mu_{\tt ST}(\alpha,\beta) (\# \cF(T))^2   \right| &\\
  \le  T^3p^{1/4+o(1)}& +T^4p^{-1/4+o(1)} ,
\end{split}
\end{equation*}
which concludes the proof.
\end{proof}

\subsection{Distribution of angles over $\cI(T)$}
\label{sec:Interv}

We start with recalling the bound from~\cite[Lemma~10]{Shp3}, which  is essentially based on
Lemma~\ref{lem:Mich bound} and the standard reduction between complete and incomplete sums
(see~\cite[Section~12.2]{IwKow}).

\begin{lemma}
\label{lem:Sin Interv}
If the polynomials $f(Z), g(Z) \in \Z[Z]$
satisfy~\eqref{eq:Nondeg},
then for any prime $p$,
 we have
$$
\sum_{\substack{t \in \cI(T)\\
\Delta(t) \not \equiv 0 \pmod p}}
\frac{\sin ((n+1)\psi_p(E(t)))}{\sin (\psi_p(E(t)))}  \ll
np^{1/2+o(1)},
$$
uniformly over all integers $n\ge 1$.
\end{lemma}

Let $\sC_{f,g,p}(\cI(T);\alpha,\beta)$ be the number of integers
$t \in \cI(T)$, where $\cI(T)$ is given by~\eqref{eq:IT},
with $\Delta(t) \not \equiv 0 \pmod p$,
such that
$$
\alpha \le \psi_p(E(t))\le \beta.
$$
Here, we reproduce the asymptotic formula  on $\sC_{f,g,p}(\cI(T);\alpha,\beta)$
given in~\cite[Lemma~11]{Shp3} with a minor change (here we use a different notation).

\begin{lemma}
\label{lem:ST Interv} If the polynomials $f(Z), g(Z) \in \Z[Z]$
satisfy~\eqref{eq:Nondeg}, then for any prime $p>T$,  we have
$$
\max_{0 \le \alpha < \beta \le \pi} \left|\sC_{f,g,p}(\cI(T);\alpha,\beta) -
\mu_{\tt ST}(\alpha,\beta) T \right| \ll T^{1/2}p^{1/4+o(1)}.
$$
\end{lemma}

\begin{proof}
Note that since $p>T$, the number of $t\in \cI(T)$ satisfying $\Delta(t) \equiv 0 \pmod p$ is upper bounded by a constant, say $c$, which only depends on the degrees of $f(Z)$ and $g(Z)$.

As in the proof of Lemma~\ref{lem:ST Farey}, by Lemma~\ref{lem:Nie} and Lemma~\ref{lem:Sin Interv}, for any positive integer $k$, we have
\begin{equation*}
\begin{split}
&\max_{0 \le \alpha < \beta \le \pi}  \left|\sC_{f,g,p}(\cI(T) ;\alpha,\beta) -
\mu_{\tt ST}(\alpha,\beta) T  \right|\\
&\qquad \ll 1 + \frac{T}{k}
+ \sum_{n=1}^k \frac{1}{n}\left|\sum_{\substack{t \in \cI(T)\\
\Delta(t) \not \equiv 0 \pmod p}}
\frac{\sin\((n+1)\psi_p(E(t))\)}{\sin\( \psi_p(E(t))\)}  \right|\\
&\qquad \ll 1 + T/k + kp^{1/2+o(1)}.
\end{split}
\end{equation*}
It is easy to see that for $T \le p^{1/2}$  the result is weaker than the
trivial bound $O(T)$.

For $T > p^{1/2}$, taking $k =\rf{p^{-1/4}T^{1/2}}$, we complete the proof.
\end{proof}

We now give yet another application of Lemma~\ref{lem:Mich bound}.

\begin{lemma}
\label{lem:Sin Setsum}
If the polynomials $f(Z), g(Z) \in \Z[Z]$
satisfy~\eqref{eq:Nondeg},
then for any non-empty subsets    $\cU, \cV \subseteq \cI(T)$ and any prime $p>T$,
 we have
$$
\sum_{\substack{u\in \cU, v \in \cV\\
\Delta(u+v) \not \equiv 0 \pmod p}}
\frac{\sin ((n+1)\psi_p(E(u+v)))}{\sin (\psi_p(E(u+v)))}  \ll
n(p\# \cU \# \cV)^{1/2},
$$
uniformly over  all
integers $n\ge 1$.
\end{lemma}
\begin{proof}
Applying the same argument as in the proof of Lemma~\ref{lem:Sin Farey},
 we have
\begin{equation*}
\begin{split}
\sum_{\substack{u\in \cU, v \in \cV\\
\Delta(u+v) \not \equiv 0 \pmod p}} &
\frac{\sin\((n+1)\psi_p(E(u+v))\)}{\sin\( \psi_p(E(u+v))\)}\\
&\ll n p^{-1/2} \sum_{m=0}^{p-1} \left|
\sum_{u\in \cU}  \ep(-mu)  \right|
\left|\sum_{ v \in \cV} \ep(-mv) \right|.
\end{split}
\end{equation*}
It now remains to apply the Cauchy inequality and note  the identities
$$
\sum_{m=0}^{p-1} \left|
\sum_{u\in \cU}  \ep(-mu)  \right|^2 = p \# \cU\quad
\text{and}
\quad
\sum_{m=0}^{p-1}\left|\sum_{ v \in \cV} \ep(-mv) \right|^2 = p \# \cV,
$$
which follow from the orthogonality of the exponential function and $p>T$.
\end{proof}

 Now, for any two non-empty subsets $\cU, \cV \subseteq \cI(T)$, let $\sD_{f,g,p}(\cU, \cV ;\alpha,\beta)$
 be the number of pairs $(u,v) \in \cU\times \cV$  with $\Delta(u+v) \not \equiv 0 \pmod p$
such that
$$
\alpha \le \psi_p(E(u+v))\le \beta.
$$
As before, combining Lemma~\ref{lem:Nie} with Lemma~\ref{lem:Sin Setsum}  we derive:

\begin{lemma}
\label{lem:ST Setsum} If the polynomials $f(Z), g(Z) \in \Z[Z]$
satisfy~\eqref{eq:Nondeg}, then for
any subsets $\cU, \cV \subseteq \cI(T)$ and any prime $p>T$,
 we have
$$
 \max_{0 \le \alpha < \beta \le \pi} \left|\sD_{f,g,p}(\cU, \cV ;\alpha,\beta) -
\mu_{\tt ST}(\alpha,\beta) \# \cU \# \cV   \right|  \ll p^{1/4}(\# \cU \# \cV)^{3/4} .
$$
\end{lemma}

\begin{proof}
Clearly, since $p>T$, we have
\begin{equation*}
\begin{split}
\#\{(u,v)  \in \cU\times \cV~:~\Delta(u+v) &\equiv 0 \pmod p\}\\
&  \ll \min\{ \# \cU,  \# \cV\} \ll (\# \cU \# \cV)^{1/2}.
\end{split}
\end{equation*}

As in the proof of Lemma~\ref{lem:ST Farey}, by Lemma~\ref{lem:Nie}, for any positive integer $k$ we have
\begin{equation*}
\begin{split}
\max_{0 \le \alpha < \beta \le \pi} & \left|\sD_{f,g,p}(\cU, \cV ;\alpha,\beta) -
\mu_{\tt ST}(\alpha,\beta) \# \cU \# \cV   \right|\\
& \ll (\# \cU \# \cV)^{1/2} + \frac{\# \cU \# \cV}{k} \\
& \qquad + \sum_{n=1}^k \frac{1}{n}\left|\sum_{\substack{u\in \cU, v \in \cV\\
\Delta(u+v) \not \equiv 0 \pmod p}}
\frac{\sin\((n+1)\psi_p(E(u+v))\)}{\sin\( \psi_p(E(u+v))\)}  \right|.
\end{split}
\end{equation*}
Thus, by Lemma~\ref{lem:Sin Setsum}, we get
\begin{equation*}
\begin{split}
\max_{0 \le \alpha < \beta \le \pi} & \left|\sD_{f,g,p}(\cU, \cV ;\alpha,\beta) -
\mu_{\tt ST}(\alpha,\beta) \# \cU \# \cV   \right|\\
& \ll (\# \cU \# \cV)^{1/2} +  \frac{\# \cU \# \cV}{k} +  k (p\# \cU \# \cV)^{1/2}\\
& \ll \frac{\# \cU \# \cV}{k} +   k (p\# \cU \# \cV)^{1/2}.
\end{split}
\end{equation*}
We can assume that $\# \cU \# \cV\ge p$, as otherwise the
result is weaker than the trivial bound $O\(\# \cU \# \cV\)$.
Then, taking $k =\rf{(p^{-1}\# \cU \# \cV)^{1/4}}$ and noticing
$$
(p^{-1}\# \cU \# \cV)^{1/4} \le  k \le (p^{-1}\# \cU \# \cV)^{1/4}+1 \le 2  (p^{-1}\# \cU \# \cV)^{1/4},
$$
we conclude the proof.
\end{proof}

\section{Proofs of Main Results}

\subsection{Proof of Theorem~\ref{thm:L-T AB}}\label{Pr:L-T AB}

From Lemma~\ref{lem:ap R}, first using the Cauchy inequality and then discarding the
conditions $\Delta(w)\not \equiv 0 \pmod p$ and $a_{w,p}
\equiv  a_p \pmod \ell$, we derive
\begin{equation}
\label{eq:LQ}
\sum_{\substack{t \in \cF(T)\\ \Delta(t) \ne 0}}
\pi_{E(t)}(\fA; x) \le P_{\cF}T^2 +  \sum_{p\le x}
L_{T,p}^{1/2} Q_{T,p}^{1/2},
\end{equation}
where
$$
L_{T,p}  =
\sum_{\substack{0\leq w \leq p-1\\
\Delta(w) \not \equiv 0 \pmod p  \\ a_{w,p}
\equiv  a_p \pmod \ell}} 1 \mand
Q_{T,p}  =  \sum_{0\leq w \leq p-1} R_{T,p}(w)^2.
$$
It is easy to see that $Q_{T,p}$ is exactly the quantity
defined in Lemma~\ref{lem:B_{T,p}}.

Therefore, for an arbitrary sequence
$\fA$, substituting the bound of Lemma~\ref{lem:B_{T,p}}
in~\eqref{eq:LQ}
and applying the bound of Lemma~\ref{lem:ap cong1} to $L_{T,p}$
with $\ell \sim x^{1/4}$, we obtain
\begin{align*}
\sum_{\substack{t \in \cF(T)\\ \Delta(t) \ne 0}}&
\pi_{E(t)}(\fA; x) \\
&\ll P_{\cF}T^2 +   \sum_{p\le x}\( x^{-1/8}p^{1/2} + x^{1/8}p^{1/4} \) \( T^2p^{-1/2} + T p^{o(1)} \) \\
&\ll Tx^{11/8+o(1)} + T^2x^{7/8}.
\end{align*}
While $\fA$ is the zero sequence, applying the bound of Lemma~\ref{lem:ap cong1} to $L_{T,p}$
 with $\ell \sim x^{1/3}$, after similar calculations
we conclude the proof.

\subsection{Proof of Theorem~\ref{thm:L-T J}}\label{Pr:L-T J}

By~\eqref{eq:pi R} and as in the proof of Theorem~\ref{thm:L-T AB}, we have
\begin{equation}\label{eq:L-T AJ}
\sum_{\substack{t \in \cF(T)\\ \Delta(t) \ne 0}}
\pi_{E(t)}(\fA; x) \le P_{\cF}T^2 +
\sum_{p\le x}
M_{T,p}^{1/2} Q_{T,p}^{1/2},
\end{equation}
where
$$
M_{T,p}  =
\sum_{\substack{0\leq w \leq p-1\\
\Delta(w) \not \equiv 0 \pmod p  \\ a_{w,p}
= a_p }} 1,
$$
and $Q_{T,p}$ is as before.

For integer $t$, we define $H(t,p)$ as the number of $\F_p$-isomorphism classes of elliptic curves over $\F_p$ with
Frobenius trace $t$.

Notice that each elliptic curve $E(w)$ has $j$-invariant $w$, which implies that each
$E(w)$ represents a distinct $\F_p$-isomorphism class of elliptic curves over $\F_p$.
So, we have
 $$
M_{T,p}  \le H(a_p,p).
$$
By~\cite[Proposition~1.9~(a)]{Lenstra1987}, for $p\ge 5$ we know that
 $$H(a_p,p)\ll p^{1/2 + o(1)},$$
 where the implied constant is independent of $p$ and $a_p$.
So, we obtain
$$M_{T,p}\ll p^{1/2 + o(1)}.$$
 Then, substituting this bound in~\eqref{eq:L-T AJ} and using the bound of
$Q_{T,p}$
from Lemma~\ref{lem:B_{T,p}}, we derive the desired result.

\subsection{Proof of Theorem~\ref{thm:L-T AB+}} \label{Pr:L-T AB+}
Here, we use a method quite different from the above.

For each $a_p$, we define two angels $\alpha_p,\beta_p\in [0,\pi]$ such that
$$
\cos\alpha_p =\min\left\{\frac{a_p}{2\sqrt{p}}+\frac{1}{p},1\right\} \quad \text{and}
\quad
\cos\beta_p =\max\left\{\frac{a_p}{2\sqrt{p}}-\frac{1}{p},-1\right\}.
$$
Then, we have
\begin{equation}
\begin{split}
\label{eq:angel}
\mu_{\tt ST}(\alpha_p,\beta_p) = \frac{2}{\pi}\int_{\alpha_p}^{\beta_p}
\sin^2\vartheta\, d \vartheta
 &=\frac{2}{\pi} \int_{\cos\beta_p}^{\cos\alpha_p}
(1-z^2)^{1/2}\, d z\\
&\le \frac{2}{\pi}(\cos\alpha_p - \cos\beta_p) \le \frac{4}{\pi p}.
\end{split}
\end{equation}

We recall the definition~\eqref{eq:def awp}
and observe that for each elliptic curve $E(t), t\in\cF(T)$ and a prime $p$,
the Frobenius trace $a_{t,p}=a_p$ if and only if
$$\cos\psi_p(E(t))=\frac{a_p}{2\sqrt{p}}.
$$
Thus, if $a_{t,p}=a_p$, we have
$$
\alpha_p \le \psi_p(E(t)) \le \beta_p.
$$

Applying the above discussions and noticing the discussion about
$N(r+s)$ and $\Delta(r+s)$ in Section~\ref{sec:not}, we get
\begin{align*}
\sum_{\substack{r,s \in \cF(T)\\ \Delta(r+s) \ne 0}}
\pi_{E(r+s)}&(\fA; x)
= \sum_{\substack{r,s \in \cF(T)\\ \Delta(r+s) \ne 0}}
\sum_{\substack{p\le x \\ p\nmid N(r+s) \\ a_{r+s,p}=a_p}} 1\\
&= \sum_{p\le x}\sum_{\substack{r,s \in \cF(T) \\ \Delta(r+s) \not\equiv 0 \pmod p
\\ a_{r+s,p}=a_p}} 1 \le \sum_{p\le x} \sB_{f,g,p}(\cF(T);\alpha_p, \beta_p),
\end{align*}
where $\sB_{f,g,p}(\cF(T);\alpha_p, \beta_p)$ has been defined in
Section~\ref{sec:Farey}.
Then, combining the above results with Lemma~\ref{lem:ST Farey}, we obtain
\begin{align*}
&\sum_{\substack{r,s \in \cF(T)\\ \Delta(r+s) \ne 0}}
\pi_{E(r+s)}(\fA; x) \\
&\qquad  \ll \sum_{p\le T} T^4
 + \sum_{T<p\le x} \(\mu_{\tt ST}(\alpha_p,\beta_p)T^4+T^3p^{1/4+o(1)}+T^4p^{-1/4+o(1)} \) \\
&\qquad\ll \sum_{p\le T} T^4 + \sum_{p\le x} \(T^4/p+T^3p^{1/4+o(1)}+T^4p^{-1/4+o(1)} \) \\
&\qquad \ll T^5+ T^4\log x + T^3x^{5/4+o(1)}+T^4x^{3/4+o(1)} \\
&\qquad \ll T^5 + T^3x^{5/4+o(1)}+T^4x^{3/4+o(1)},
\end{align*}
which completes the proof.

\subsection{Proof of Theorem~\ref{thm:L-T K}} \label{Pr:L-T K}

As Lemma~\ref{lem:ap R} and using the Cauchy inequality, we obtain
\begin{align*}
\sum_{\substack{t \in \cF(T)\\ \Delta(t) \ne 0}}
\pi_{E(t)}(\K; x)
& \le P_{\cF}T^2 +
\sum_{p\le x}
 \sum_{\substack{0\leq w \leq p-1\\
\Delta(w) \not \equiv 0 \pmod p  \\ a_{w,p}\ne 0 \\ \Q\(\sqrt{a_{w,p}^2-4p}\)=\K}} R_{T,p}(w)\\
&\le P_{\cF}T^2 +  \sum_{p\le x}N_{T,p}^{1/2}Q_{T,p}^{1/2},
\end{align*}
where
$$
N_{T,p}  =
\sum_{\substack{0\leq w \leq p-1\\
\Delta(w) \not \equiv 0 \pmod p  \\ a_{w,p}\ne 0 \\ \Q\(\sqrt{a_{w,p}^2-4p}\)=\K}} 1,
$$
and $Q_{T,p}$ is as before.

Applying the bound of Lemma~\ref{lem:ap cong2} to $N_{T,p}$ with $\ell \sim x^{1/3}$, we obtain
\begin{equation}
\label{eq:NT bound}
N_{T,p} \ll p x^{-1/3}+ x^{1/6}p^{1/2} .
\end{equation}
Now, using the bound of $Q_{T,p}$
from Lemma~\ref{lem:B_{T,p}}, we obtain
\begin{align*}
\sum_{\substack{t \in \cF(T)\\ \Delta(t) \ne 0}}
&\pi_{E(t)}(\K; x) \\
& \ll T^2 +
\sum_{p\le x} \(p^{1/2} x^{-1/6}+ x^{1/12}p^{1/4}\)
\(T^2p^{-1/2} + Tp^{o(1)}\),
\end{align*}
and after simple calculations, we complete the proof.

\subsection{Proof of Theorem~\ref{thm:L-T K+}} \label{Pr:L-T K+}

For an integer $w$, we denote by $U_{T,p}(w)$ the number
of pairs $(u_1/v_1,u_2/v_2) \in \cF(T) \times \cF(T)$ with $\gcd(v_1v_2,p)=1$ and $u_1/v_1 + u_2/v_2 \equiv w \pmod p$.

As in the proof of Theorem~\ref{thm:L-T K}, we obtain
\begin{align*}
\sum_{\substack{r,s \in \cF(T)\\ \Delta(r+s) \ne 0}}
\pi_{E(r+s)}(\K; x)
&\le P_{\cF}T^4 +
\sum_{p\le x}
 \sum_{\substack{0\leq w \leq p-1\\
\Delta(w) \not \equiv 0 \pmod p  \\ a_{w,p}\ne 0 \\ \Q\(\sqrt{a_{w,p}^2-4p}\)=\K}} U_{T,p}(w)\\
&\le P_{\cF}T^4 +  \sum_{p\le x}N_{T,p}^{1/2}V_{T,p}^{1/2},
\end{align*}
where
$$
N_{T,p}  =
\sum_{\substack{0\leq w \leq p-1\\
\Delta(w) \not \equiv 0 \pmod p  \\ a_{w,p}\ne 0 \\ \Q\(\sqrt{a_{w,p}^2-4p}\)=\K}} 1,
$$
and  $V_{T,p}$ is  as in Section~\ref{sec:cong FT}.
Applying the bound of Lemma~\ref{lem:ap cong2} to $N_{T,p}$ with $\ell \sim x^{1/3}$, we again obtain the bound~\eqref{eq:NT bound} from which we conclude
that  $N_{T,p} \ll  x^{2/3}$. Hence
$$
\sum_{\substack{r,s \in \cF(T)\\ \Delta(r+s) \ne 0}}
\pi_{E(r+s)}(\K; x)
  \ll  T^4 + x^{1/3} \sum_{p\le x} V_{T,p}^{1/2}.
$$
Using   Lemma~\ref{lem:V_{T,p}}, we
derive
$$
 \sum_{p\le x} V_{T,p}^{1/2}   \ll
T^{4}x^{1/2} + T^{2+o(1)}x^{1+o(1)} .
$$
Hence
$$
\sum_{\substack{r , s\in \cF(T) \\ \Delta(r+s) \ne 0}}
\pi_{E(r+s)}(\K; x)
\ll  T^4x^{5/6} + T^{2+o(1)} x^{4/3+o(1)}.
$$
Clearly we can assume that $T \ge x^{1/6}$ as otherwise
the result is weaker than the trivial bound $O(T^4x)$.
In this case replace $T^{2+o(1)} x^{4/3+o(1)}$ with
$T^{2+o(1)} x^{4/3}$ and the result  follows.

\subsection{Proof of Theorem~\ref{thm:S-T Farey}} \label{Pr:S-T Farey}

Using the same notation as in Section~\ref{sec:Farey} and noticing the discussion about $N(r+s)$ and $\Delta(r+s)$ in Section~\ref{sec:not}, we have
\begin{align*}
\sum_{\substack{r , s\in \cF(T) \\ \Delta(r+s) \ne 0}}
\pi_{E(r+s)}&(\alpha,\beta; x)
=\sum_{\substack{r , s\in \cF(T) \\ \Delta(r+s) \ne 0}}
\sum_{\substack{p\le x \\ p \nmid N(r+s)\\ \psi_p(E(r+s))\in [\alpha,\beta]}} 1\\
&= \sum_{p\le x}
\sum_{\substack{r , s\in \cF(T) \\
\Delta(r+s) \not\equiv 0 \pmod p \\ \psi_p(E(r+s))\in [\alpha,\beta]}} 1
= \sum_{p\le x}\sB_{f,g,p}(\cF(T) ;\alpha,\beta).
\end{align*}
By Lemma~\ref{lem:ST Farey}, we get
\begin{align*}
  \sum_{\substack{r , s\in \cF(T) \\ \Delta(r+s) \ne 0}} &
\pi_{E(r+s)}(\alpha,\beta; x)-\sum_{p\le x} \mu_{\tt ST}(\alpha,\beta) (\# \cF(T))^2 \\
&\qquad \ll \sum_{p\le T} T^4 + \sum_{T<p\le x} \( T^3p^{1/4+o(1)}+T^4p^{-1/4+o(1)} \) \\
&\qquad \ll T^5 + T^3x^{5/4+o(1)} + T^4x^{3/4+o(1)}.
\end{align*}
Thus, the desired result follows from~\eqref{eq: Farey} and the assumption
$x^{1/4+\eps} \le T \le x^{1-\eps}$.

\subsection{Proof of Theorem~\ref{thm:L-T I}} \label{Pr:L-T I}

As in Section~\ref{Pr:L-T AB+},  we have
\begin{align*}
\sum_{\substack{t \in \cI(T)\\ \Delta(t) \ne 0}}
\pi_{E(t)}&(\fA; x)
= \sum_{\substack{t \in \cI(T)\\ \Delta(t) \ne 0}}
\sum_{\substack{p\le x \\ p\nmid N(t) \\ a_{t,p}=a_p}} 1\\
&= \sum_{p\le x}\sum_{\substack{t \in \cI(T) \\ \Delta(t) \not\equiv 0 \pmod p
\\ a_{t,p}=a_p}} 1 \le \sum_{p\le x} \sC_{f,g,p}(\cI(T);\alpha_p, \beta_p),
\end{align*}
where $\alpha_p$ and $\beta_p$ have been defined in Section~\ref{Pr:L-T AB+}, and
$\sC_{f,g,p}(\cI(T);\alpha_p, \beta_p)$ has been defined in Section~\ref{sec:Interv}.
Then, combining the above inequality with Lemma~\ref{lem:ST Interv} and~\eqref{eq:angel}, we obtain
\begin{align*}
&\sum_{\substack{t \in \cI(T)\\ \Delta(t) \ne 0}}
\pi_{E(t)}(\fA; x) \\
&\qquad \ll \sum_{p\le T} T + \sum_{T<p\le x} \(\mu_{\tt ST}(\alpha_p,\beta_p)T+T^{1/2}p^{1/4+o(1)} \)\\
&\qquad \ll T^2+\sum_{p\le x} \(T/p+T^{1/2}p^{1/4+o(1)}\)\\
&\qquad \ll T^2+T\log x + T^{1/2}x^{5/4+o(1)}.
\end{align*}
Noticing that $T\log x \le \sqrt{ T^2 \cdot T^{1/2}x^{5/4+o(1)}}$, we conclude the proof.

\subsection{Proof of Theorem~\ref{thm:L-T Setsum}} \label{Pr:L-T Setsum}

As in Section~\ref{Pr:L-T AB+}, we obtain
\begin{align*}
\sum_{\substack{u \in \cU, v\in \cV \\ \Delta(u+v) \ne 0}}
\pi_{E(u+v)}&(\fA; x)
= \sum_{\substack{u \in \cU, v\in \cV\\ \Delta(u+v) \ne 0}}
\sum_{\substack{p\le x \\ p\nmid N(u+v) \\ a_{u+v,p}=a_p}} 1\\
&= \sum_{p\le x}\sum_{\substack{u \in \cU, v\in \cV \\ \Delta(u+v) \not\equiv 0 \pmod p
\\ a_{u+v,p}=a_p}} 1 \le \sum_{p\le x} \sD_{f,g,p}(\cU, \cV;\alpha_p, \beta_p),
\end{align*}
where $\alpha_p$ and $\beta_p$ are as the above, and
$\sD_{f,g,p}(\cU, \cV;\alpha_p, \beta_p)$ has been defined in Section~\ref{sec:Interv}.

By Lemma~\ref{lem:ST Setsum} and the bound~\eqref{eq:angel},
we get
\begin{align*}
&\sum_{\substack{u \in \cU, v\in \cV \\ \Delta(u+v) \ne 0}}
\pi_{E(u+v)}(\fA; x)\\
&\quad \ll \sum_{p\le T} \#\cU \#\cV  + \sum_{T<p\le x} (\mu_{\tt ST}(\alpha_p,\beta_p)\#\cU \#\cV+
 p^{1/4}(\#\cU \#\cV)^{3/4})\\
&\qquad\ll T\#\cU \#\cV   + \sum_{p\le x} \(\#\cU \#\cV/p+p^{1/4}(\#\cU \#\cV)^{3/4}\)\\
&\qquad\ll T\#\cU \#\cV + \#\cU \#\cV\log x + (\#\cU \#\cV)^{3/4}x^{5/4} .
\end{align*}
 We now note that the second term never dominates and
can be removed. Indeed, since  $\#\cU \#\cV  \le T^2$, we have
for the geometric mean of the first and the third terms:
$$
\sqrt{T\#\cU \#\cV \cdot (\#\cU \#\cV)^{3/4} x^{5/4}}
\ge  (\#\cU \#\cV)^{9/8}  x^{5/8}
\gg \#\cU \#\cV  \log x.
$$
This completes the proof.

\subsection{Proof of Theorem~\ref{thm:L-T KI}}
 \label{Pr:L-T KI}

The proof of Theorem~\ref{thm:L-T KI}
is almost the same as that  of Theorem~\ref{thm:L-T K} in Section~\ref{Pr:L-T K},
and in fact, is  simpler, because the parameter $t\in \cI(T)$ is an integer.

We only need to note that
the number of solutions to the congruence
\begin{equation*}
t_1   \equiv t_2 \pmod p, \qquad t_1,t_2 \in \cI(T),
\end{equation*}
is upper bounded by $O(T+T^2/p)$.
Then, as in the proof of Theorem~\ref{thm:L-T K}, we have
\begin{align*}
\sum_{\substack{t\in \cI(T)\\ \Delta(t) \ne 0}}
\pi_{E(t)}(\K; x)
&\ll \sum_{p\le x}(x^{-1/6}p^{1/2}+x^{1/12}p^{1/4})(T^{1/2}+Tp^{-1/2}) \\
& \ll T^{1/2}x^{4/3} + Tx^{5/6},
\end{align*}
which concludes the proof.

\subsection{Proof of Theorem~\ref{thm:S-T Setsum}} \label{Pr:S-T Setsum}

Using the notation in Section~\ref{sec:Interv} and noticing the discussion about $N(u+v)$ and $\Delta(u+v)$ in Section~\ref{sec:not}, we have
\begin{align*}
\sum_{\substack{u \in \cU, v\in \cV \\ \Delta(u+v) \ne 0}}
\pi_{E(u+v)}&(\alpha,\beta; x)
=\sum_{\substack{u \in \cU, v\in \cV \\ \Delta(u+v) \ne 0}}
\sum_{\substack{p\le x \\ p \nmid N(u+v)\\ \psi_p(E(u+v))\in [\alpha,\beta]}} 1\\
&= \sum_{p\le x}
\sum_{\substack{u \in \cU, v\in \cV \\ \Delta(u+v) \not\equiv 0 \pmod p \\ \psi_p(E(u+v))\in [\alpha,\beta]}} 1 = \sum_{p\le x}\sD_{f,g,p}(\cU, \cV ;\alpha,\beta).
\end{align*}

By Lemma~\ref{lem:ST Setsum}, we get
\begin{align*}
 \sum_{\substack{u \in \cU, v\in \cV \\ \Delta(u+v) \ne 0}}
\pi_{E(u+v)}(\alpha&,\beta; x)-\sum_{p\le x} \mu_{\tt ST}(\alpha,\beta) \# \cU \# \cV \\
&   \ll \sum_{p\le T} \# \cU \# \cV + \sum_{T<p\le x} p^{1/4}(\# \cU \# \cV)^{3/4} \\
& \ll \pi(T)\# \cU \# \cV + \pi(x) x^{1/4}(\# \cU \# \cV)^{3/4}.
\end{align*}
Then, the desired result follows from the assumptions $\# \cU \# \cV\ge x^{1+\eps}$
and $T\le x^{1-\eps}$.

\section*{Acknowledgements}
The authors would like to thank the referee for careful reading and valuable comments. The research of the authors was supported by the Australian Research Council Grant DP130100237.

\end{document}